\documentclass[11pt]{amsart}
\usepackage[T1]{fontenc}
\usepackage{ dsfont}
\usepackage{amsmath,amsthm, amssymb, latexsym,}
\usepackage{cases}
\usepackage{mathrsfs}
\usepackage{hyperref}
\usepackage{pifont}
\usepackage{amsfonts}
\usepackage{stmaryrd}
\usepackage{color}
\usepackage[all]{xy}
\usepackage[english]{babel}
\usepackage[latin1]{inputenc}

\usepackage[all]{xy}
\newtheorem{theorem}{Theorem}[section]
\newtheorem{lemma}[theorem]{Lemma}
\newtheorem{corollary}[theorem]{Corollary}
\newtheorem{proposition}[theorem]{Proposition}

\theoremstyle{definition}
\newtheorem{definition}[theorem]{Definition}

\theoremstyle{remark}
\newtheorem{remark}[theorem]{Remark}

\numberwithin{equation}{section}

\newcommand{\on}[1]{\operatorname{#1}}
\newcommand{\abs}[1]{\lvert#1\rvert}

\newcommand{\mc}[1]{\mathcal{#1}} %mathcal, use capitals
\newcommand{\mf}[1]{\mathfrak{#1}} %mathfrak, both capitals and ordinary
\newcommand{\bd}{\bullet} % big dot, using in complexes

\newcommand{\lieg}{\mathfrak{g}} % lie algebra g
\newcommand{\ten}{\otimes} % tensor
\newcommand{\bds}{\bigoplus} % big direct sum
\newcommand{\bten}{\bigotimes} % big tensor

\newcommand{\fk}{\mathbb{K}} % field symbol of natural number N
\newcommand{\ff}{\mathbb{F}} % field symbol of arbitrary F
\newcommand{\ra}{\rightarrow} % "right arrow" in a map
\newcommand{\iso}{\simeq} % isomorphism
\newcommand{\pr}[1]{#1^{\prime}}

\newcommand{\GL}{\on{GL}}
\begin{document}
\title[Block Basis]{Block Basis for Coinvariants of Modular Pseudo-reflection Groups}
\author{Ke Ou} \address[Ke Ou]{School of Statistics and Mathematics, Yunnan University of Finance and Economics, Kunming 650221, China} \email{keou@ynufe.edu.cn}
\date{\today}
\begin{abstract}
	As a sequel of \cite{Ou}, in this shot note, we investigate block basis for coinvariants of finite modular pseudo-reflection groups.
	We are particularly interested in the case where $ G $ is a subgroup of the parabolic subgroups of $ \GL_n(q) $ which generalizes the Weyl groups of restricted Cartan type Lie algebra.	
\end{abstract}
\maketitle
\section{Introduction}
Let $ p $ be a fixed prime and $ \ff_q $ be the finite field with $ q = p^r $ for some $ r\geq 1. $ The finite
general linear group $ \GL_n(q) $ acts naturally on the symmetric algebra $ \mc{P}:=S^{\bd}(V) $ where $ V=\ff_q^n $ is the natural $ \GL_n(q) $ -module. 
If $ G $ is any subgroup of $ \GL_n(q), $ then we denote by $ \mc{P}^G $ the ring of $ G $-invariant polynomials. The ring of coinvariants, which we denote by $ \mc{P}_G, $ is the quotient of $ \mc{P} $ by the
ideal generated by the homogeneous elements in $ \mc{P}^G $ with positive degree.

A block basis for $ \mc{P}_G $ (see \ref{def} for definition) is a nice basis consisting of the monomial factors of a single monomial. Such basis simplify computations in many ways. Generally, the combinatorics involved in writing a polynomial or the products of elements in the block basis in terms of such a basis is less complicated since we are dealing with monomials. As an application, the authors of \cite{CHSW} use the block basis to determine the image of the transfer over some specific group G. 

The $ \GL_n(q) $ invariants in $ \mc{P}$ are determined by Dickson \cite{Di}. The block basis for $ \mc{P}_{\GL_n(q)} $ is investigated by Steinberg \cite{St} and Campell-Hughes-Shank-Wehlau \cite{CHSW}.
For a composition $ I=(n_1,\cdots,n_l) $ of $ n, $ let $ \GL_I $ be the parabolic subgroup associated to $ I $. Generalizing \cite{Di}, Kuhn and Mitchell \cite{KM} showed that the algebra $ \mc{P}^{\GL_I} $ is a polynomial algebra in $ n $ explicit generators. 

Let $ G_I $ and $ U_I $ be a subgroup of $ \GL_I $ with forms  
\begin{equation}\label{G}
G_I=\left( 
\begin{matrix}G_1 & * & \cdots & *\\ 
0 & G_2 &\cdots & *\\
\vdots & \vdots & \vdots & \vdots\\
0 & 0 & \cdots & G_l
\end{matrix}
\right)
\text{ and }
U_I=\left( 
\begin{matrix}I_{n_1} & * & \cdots & *\\ 
0 & I_{n_2} &\cdots & *\\
\vdots & \vdots & \vdots & \vdots\\
0 & 0 & \cdots & I_{n_l}
\end{matrix}
\right) \text{ respectively,}
\end{equation}
such that $  G_i< \GL_{n_i}(q) $ for all $ i $ where $ I_j $ is the identity matrix of $ \GL_j(q). $

Note that $ G_I $ is a  generalization of  $ \GL_I $ as well as the Weyl groups of Cartan type Lie algebras. Precisely, $ G_I=\GL_I $ if $ G_i= \GL_{n_i}(q) $ for all $ i. $ 
And $ G_I $ is a Weyl group of Cartan type Lie algebras if $ l=2,\ q=p,$  $ G_1=\GL_{n_1}(q),\ G_2=S_{n_2} $ or $ G(2,1,n_2) $ (cf. \cite{Je}). From the viewpoint of representation theory, the coinvariants of Weyl group of Lie algebra $ \lieg $ are providing very interesting yet limited answers to the problem of understanding $ \lieg $ modules. For example, Jantzen describes the basic algebra of block with regular weight by using coinvariants of Weyl group in \cite[Proposition 10.12]{Jan1}. 

Meanwhile, $ G_I $ is a modular finite pseudo-reflection group  if $ l\geq 2 $ and all $ G_i $ are pseudo-reflection groups since $ p\mid |U_I| $ (subsection \ref{PRG}). The invariants and coinvariants for a modular pseudo-reflection group can be quite complicate (see \cite{Na} for example). Our investigation generalizes the results for modular coinvariants by Steinberg \cite{St} and Campell-Hughes-Shank-Wehlau \cite{CHSW}.

In \cite{Ou}, we study $ \mc{P}^{U_I} $ and $ \mc{P}^{G_I} $ which are in turns out to be polynomial algebras. In this short note, we are interested in the block basis for the coinvariants $ \mc{P}_{U_I} $ (Theorem \ref{U_I}) and $ \mc{P}_{G_I} $ (Theorem \ref{6.4}). As an application, the block basis for $ \mc{P}_{G_I}$ are described precisely when $ G_i $ is $ \GL_{n_i}(q)$ and $ G(m,a,n_i) $ in Theorem \ref{5.6}, Proposition \ref{6.5} and Proposition \ref{6.7}. Our approach is in turn built on \cite{CHSW} where the authors describe sufficient conditions for the existence of a block basis for $ \mc{P}_G. $

The paper is organized as follows. In Section 2 and 3, we review some needed results from \cite{CHSW,Di,He,Ka,KM,MT,Ou}. Section 4 deals with the block basis for coinvariants of $ {G(m,a,n)} $. A block basis for $ \mc{P}_{U_I} $ are given in Section 5 and section 6 describes the block basis for $ \mc{P}_{G_I}. $

\section{Preliminary}
\subsection{}
Set $ m_0=0$ and $ m_k=\sum_{i=1}^{k}n_i,\ k=1,\cdots,l. $ 
For each $ 1\leq s\leq n, $ define $$ \tau(s)=m_j \text{ if } m_j<s\leq m_{j+1}. $$

Recall the definition of $ G_I $ and $ U_I $ in \ref{G}. Then $ G_I=L_I\ltimes U_I, $  where \[
L_I=\left( 
\begin{matrix}G_1 & 0 & \cdots & 0\\ 
0 & G_2 &\cdots & 0\\
\vdots & \vdots & \vdots & \vdots\\
0 & 0 & \cdots & G_l
\end{matrix}
\right),
\]

Suppose $ V=\langle x_1,\cdots,x_n\rangle_{\ff_q}, $ the symmetric algebra  $ S^\bd (V) $ will be identified with $ \ff_q[x_1,\cdots, x_n] $. Namely, $ \mc{P}=\ff_q[x_1,\cdots, x_n]. $ 

\subsection{Pseudo-reflection groups}\label{PRG}
In this subsection, we will recall some basic facts for pseudo-reflection groups. More details refer to \cite{Ka}.

For a finite dimensional vector space $ W $ over $ \ff_q $, a pseudo-reflection is a linear isomorphism $ s: W \ra W $ that is not the
identity map, but leaves a hyperplane $  H\subseteq W $ pointwise invariant. $ G \subseteq \GL(W) $ is a pseudo-reflection group if $ G $ is generated by its
pseudo-reflections. We call $ G $ is non-modular if $ p\not| \abs{G} $ while $ G $ is modular otherwise. 
\begin{lemma}\cite[Lemma 2.1]{Ou}
	$ G_I $ is a finite pseudo-reflection group if all $ G_i $ are finite pseudo-reflection groups.
\end{lemma}
\begin{remark}
	As a corollary, the Weyl groups of restricted Cartan type Lie algebras with type $ W,\ S $ and $ H $ are modular pseudo-reflection groups (\cite{Je}).
\end{remark}

\subsection{block basis}
In this subsection, we will list some results of block basis for coinvariants. More details refer to \cite{CHSW}.

A homogeneous system of parameters for $ \mc{P}$ is any collection of homogeneous
elements $ {e_1,\cdots, e_n} $ with the property that they generate a polynomial
subalgebra $ \mc{A} = \fk[ f_1,\cdots,f_n] $ over which $ \mc{P}$ is a free $ \mc{A} $-module. Moreover, a basis for $ \mc{P}$ as a free module over $ \mc{A} $ is any set of elements of $ \mc{P}$ that projects to a $ \fk $-basis for $ P/\mc{I} $ where $ \mc{I}=(\mc{A}_+) $ is the ideal generated by $ \{ f_1,\cdots,f_n \}. $ Denote $ d_i $ the degree of $ f_i $ for $ i=1,\cdots ,n. $

\begin{definition}\cite[Definition 1.1]{CHSW}\label{def}
	Let $ f\in \mc{P}$ be a monomial. We say \textit{$ f $ generates a block basis for $ \mc{P}$ over $ \mc{A} $} (or {\it for $ \mc{P}/\mc{I} $}) if the set of all monomial factors of $ f $ is a basis of vector space $ \mc{P}/\mc{I}. $ Such a basis consisting of all the monomial factors of a single monomial is called a \textit{block basis}.
\end{definition}

In \cite{CHSW}, the authors provide sufficient conditions for the existence of a
block basis for $ P/\mc{I}. $ Note that the symmetric group $ S_n $ acts on $ \mc{P}$ by
permuting the variables. Let $ \Sigma (\mc{A}) $ be the subgroup of $ S_n $ fixing
$ \mc{A} $ pointwise. Denote $ a\equiv b $ if $ a-b\in \mc{I} $.

\begin{definition}\cite[Definition 4.1]{CHSW}
	We will say that $ w = x_1^{m_1}\cdots x_n^{m_n} $ is a critical monomial
	associated to $ \alpha=x_1^{t_1}\cdots x_n^{t_n}, $ if:
	\begin{enumerate}
		\item for every $ g \in \Sigma(\mc{A}), $ there is an $ i $ such that $ m_i > t_{g(i)} $;
	   \item  no proper factor of $ w $ satisfies the first condition.
\end{enumerate}
\end{definition}

Suppose $ \Sigma(\mc{A})=S_n $ and the $ t_i $ form a decreasing sequence. In this case, the critical monomials associated to $ \alpha $ are $ (x_1\cdots x_i)^{t_i+1}. $

\begin{lemma}\cite[Theorem 4.1]{CHSW}\label{critical}
	Suppose $ \alpha=x_1^{t_1}\cdots x_n^{t_n} $ generates a block basis. Then the sequence $ t_i $ is a permutation of the sequence $ d_i-1 $ and the degree of $ \alpha $ is $ \sum_{i=1}^nd_i-n. $
	Furthermore, if $ g\in\Sigma(\mc{A}), $ then $ g(\alpha) $ also generates a block basis.
\end{lemma}

\begin{lemma}\cite[Theorem 4.4]{CHSW}\label{critical2}
	Suppose $ \alpha=x_1^{t_1}\cdots x_n^{t_n} $ has trivial critical monomials and the sequence $ t_i $ is a permutation of $ d_i - 1. $ Then $ \alpha $ generates a block basis for $ P/\mc{I}. $
\end{lemma}

\iffalse
\begin{lemma}\cite[Theorem 4.7]{CHSW}\label{2.6}
	Suppose that $ \Sigma(\mc{A}) = {1}. $ Then the critical monomials
	associated to $ \alpha $ are $ x_i^{d_i}. $ If these critical monomials are trivial, then $ \mc{P}/\mc{I}\iso \otimes_{i=1}^n \fk[x_i]/(x_i^{d_i}) $ as  algebras.
\end{lemma}
\fi

\section{invariants of $ \mc{P}$}
In this section, we will first recall the works by Dickson \cite{Di} and Kuhn-Mitchell \cite{KM} on invariants in $ P. $ And then tcretirierhe $ G_I $ invariants in $ \mc{P}$ will be investigated. 

\subsection{The invariants of Dickson and Kuhn-Mitchell}
For $ 1\leq k\leq n, $ define homogeneous polynomials $ V_k,  L_n,\ Q_{n,k} $ as follows:
\[ V_k= \prod_{\lambda_1,\cdots, \lambda_{{k-1}} \in\ff_q} (\lambda_1x_1+\cdots\lambda_{k-1}x_{k-1}+x_k), \]
\[ L_k=\prod_{i=1}^k V_i = \prod_{i=1}^k \prod_{\lambda_1,\cdots, \lambda_{{i-1}} \in\ff_q} (\lambda_1x_1+\cdots\lambda_{i-1}x_{i-1}+x_i), \]
\[ F_n(X)=\prod_{\lambda_1,\cdots, \lambda_{{i-1}} \in\ff_q} (X+\lambda_1x_1+\cdots\lambda_{n}x_{n})=X^{q^n}+\sum_{k=0}^{n-1} Q_{n,k}X^{q^k}. \]

According to \cite{Di}, the algebra of invariants over $ \GL_n(q) $ in $ \mc{P} $ are polynomial algebras. Moreover,
\begin{equation}\label{Dickson2}
\mc{P}^{\GL_n(q)} = \ff_q[Q_{n,0}, \cdots, Q_{n,n-1}].\quad\ \ 
\end{equation}
		
For $ 1\leq i\leq l,\ 1\leq j\leq n_i, $ define 
\begin{equation}\label{vij}
	v_{i,j}=F_{m_{i-1}}(x_{m_{i-1}+j}),
\end{equation}
\begin{equation}\label{qij}
	q_{i,j}=Q_{n_i,j}(v_{i,1},\cdots,v_{i,n_i} ).
\end{equation}
Then $ \deg(v_{i,j})=q^{m_{i-1}} $ and $ \deg (q_{i,j})=q^{m_i}-q^{m_i-j}. $ %By definition, $ v_{i,j}=L_{i+1}(x_1,\cdots,x_) $

By the proof of \cite[Lemma 1]{MT},	
\begin{equation}\label{U_I on A}
\mc{P}^{U_I} = \ff_q[x_1,\cdots,x_{n_1}, v_{2,1} ,\cdots, v_{2,n_2},\cdots, v_{l,1},\cdots, v_{l,n_l} ].  
\end{equation}
Moreover, by \cite[Theorem 2.2]{KM} and \cite[Theorem 1.4]{He},
\begin{equation}\label{KM}
\mc{P}^{\GL_I} = \ff_q[q_{i,j}\mid 1\leq i\leq l, 1\leq j\leq n_i],
\end{equation}

\subsection{The invariants of $ G_I $}
\begin{lemma}\cite[Proposition 3.2]{Ou}
	For $ 1\leq i\leq l,  $
	assume that $ \ff_q[x_1,\cdots,x_{n_i}]^{G_{i}} =\ff_q[e_{i,1},\cdots, e_{i, n_i}] $ is a polynomial algebra. For $ 1\leq j\leq n_i, $ define
	$  u_{i,j}=e_{i,j}(v_{i,1},\cdots,v_{i,n_i}).  $ The subalgebra $ \mc{P}^{G_I} $ of $ G_I $-invariants in $ \mc{P}$ is a polynomial ring on the generators $ u_{i,j} $  with  $ 1\leq i\leq l, 1\leq j\leq n_i. $ Namely,
	$$  P^{G} = \ff_q[u_{i,j}\mid 1\leq i\leq l, 1\leq j\leq n_i]. $$
\end{lemma}

\section{block basis for the coinvariants of $ G(m,a,n) $}
Recall that $ G(m,a,n)\iso S_{n}\ltimes A(m,a,n), $ which is called imprimitive reflection group, where
\[ A(m,a,n)=\{ \on{diag}(w_1,\cdots,w_{n})\mid w_j^m=(w_1\cdots w_{n})^{m/a}=1 \}. \]

\iffalse
Since $ G_i<\GL_{n_i}(q), $ one can check directly that
\[ G(m,a,n_i)=G(\pr{m},\pr{a},n_i), \text{ where } \pr{m}=(q-1,m),\ \pr{a}=\pr{m}/(q-1,m/a). \]
Moreover, assume that $ m\mid (q-1). $ 
\fi

Denote $ G=G(m,a,n), $ it is well-known that
$ \mc{P}^G=\fk[{e_1},\cdots,{e_n}], $ where $$ e_i=
\left\{ \begin{matrix}
\sum_{1\leq j_1<\cdots<j_i\leq n } x_{j_1}^m\cdots x_{j_i}^m & i\neq n\\
(x_1\cdots x_n)^{m/a} & i=n
\end{matrix} 
\right.. $$ 
Recall that $ \mc{I} $ is the ideal of $ \mc{P}$ generated by $ e_1,\cdots,e_n, $ and $ d_i:=\deg (e_i)= \left\{ \begin{matrix}
mi & i\neq n\\
mn/a & i=n
\end{matrix} 
\right.. $ Clearly $ \Sigma(\mc{P}^G)=S_n. $ 
\subsection{$ a=1 $ case} In this case,
let $ \alpha=x_1^{m-1}x_2^{2m-1}\cdots x_n^{mn-1}. $ The critical monomials associated to $ \alpha $ are $ (x_n\cdots x_i)^{mi}. $ For each $ 1\leq i\leq n $ and $ 0\leq k\leq i, $ let $ e_{k,i} =\sum_{1\leq j_1<\cdots<j_k\leq i } x_{j_1}^m\cdots x_{j_k}^m $ if $ k\neq 0 $ and $ e_{0,i}=0. $ Then $ e_{k,n}=e_k  $ for all $ k=1,\cdots,n. $

\begin{lemma}
	For all $ 1\leq i\leq n $ and $ 1\leq r\leq n+1-i, $ one have\[ (x_n\cdots  x_{n+1-i})^{mr}e_{n+1-i-r,n-i}\in \mc{I}. \]
\end{lemma}
\begin{proof}
	The proof is by induction on $ i. $
	\begin{enumerate}
		\item When $ i=1, $ one need to verify that $ x_n^{mr}e_{n-r,n-1}\in \mc{I}. $ We will use induction on $ r $ to prove it.
		\begin{enumerate}
			\item When $ r=1, $ $ x_n^me_{n-1,n-1}=e_{n,n}\in \mc{I}. $
			\item Suppose $ r>1. $ Note that
			\begin{equation}\label{5.1}
			e_{k,i}=e_{k,i-1}+x_i^me_{k-1.i-1}.
			\end{equation}
		\end{enumerate}
		
		Therefore, $ x_n^{mr}e_{n-r,n-1}+ x_n^{m(r-1)}e_{n-r+1,n-1}=x_n^{m({r-1})} e_{n-r+1}\in \mc{I}. $ By induction assumption on $ r,$  $ x_n^{m(r-1)}e_{n-r+1,n-1}\in \mc{I}. $ Thus $ x_n^{mr}e_{n-r,n-1}\in \mc{I}. $
		\item Suppose $ i>1. $ Again we will use induction on $ r $ to prove $ (x_n\cdots x_{n+1-i})^{mr}e_{n+1-i-r,n-i}\in \mc{I}. $
		\begin{enumerate}
			\item When $ r=1, $ $ (x_n\cdots x_{n+1-i})^{m}e_{n-i,n-i}=e_{n,n}\in \mc{I}. $
			\item Suppose $ r>1. $ By \ref{5.1}
			$$ (x_n\cdots x_{n+1-i})^{mr}e_{n+1-i-r,n-i} + (x_n\cdots x_{n+2-i})^{mr}x_{n+1-i}^{m({r-1})}e_{n+2-i-r,n-i}$$
			$$=(x_n\cdots x_{n+2-i})^{mr}x_{n+1-i}^{m({r-1})}e_{n+2-i-r,n-i+1}. $$ 
		\end{enumerate}
		
		Using the primary induction hypothesis, $ (x_n\cdots x_{n+2-i})^{mr}x_{n+1-i}^{m({r-1})}e_{n+2-i-r,n-i+1}\in \mc{I}. $
		Using the secondary induction hypothesis, $ (x_n\cdots x_{n+2-i})^{mr}x_{n+1-i}^{m({r-1})}e_{n+2-i-r,n-i}\in \mc{I}. $ Thus, $ (x_n\cdots x_{n+1-i})^{mr}e_{n+1-i-r,n-i}\in \mc{I}. $
	\end{enumerate}
\end{proof}

Set $ r=n+1-i, $ one have $ (x_n\cdots  x_{r})^{mr}e_{0,n-i}=(x_n\cdots  x_{r})^{mr}\in \mc{I}. $
\begin{corollary}
	The critical monomial $ (x_n\cdots  x_{r})^{mr}, r=1,\cdots,n, $ are trivial.
\end{corollary}

By Lemma \ref{critical2}, one have
\begin{corollary}\label{1}
	The monomial $ \alpha=x_1^{m-1}x_2^{2m-1}\cdots x_n^{mn-1} $ generates a block basis for the coinvariants of $ G(m,1,n). $
\end{corollary} 

By Lemma \ref{critical}, it's convenient to deal with the block basis given by $ x_1^{mn-1}x_2^{m(n-1)-1}\cdots x_n^{m-1} $ to describe the product in $ \mc{P}_G. $ In fact, we need to write $ x_{n+1-r}^{mr} $ in terms of our block basis.

Let $ s_{k,i}=\sum_{j_1+\cdots+j_i=k} x_{j_1}^{mj_1}\cdots x_{j_i}^{mj_i}. $ Observe that  $s_{k,i}=x_i^ms_{k-1,i}+s_{k,i-1}. $ The proof of \cite[Theorem 8.4]{CHSW} still works. Namely, if $ k+i\geq n+1, $ $  s_{k,n+1-i}\in \mc{I}.  $ As Corollary, one have
\begin{proposition}\label{5.4}
	Keep notations as above,
	$$ x_{n+1-r}^{mr}\equiv-\sum_{\stackrel{j_1+\cdots+j_{n+1-r}= r}{j_{n+1-r}\neq r}} x_{j_1}^{mj_1}\cdots x_{j_{n+1-r}}^{mj_{n+1-r}}. $$
\end{proposition}

Note that the summands on the right hand side form a subset of our block basis. This proposition provides a recursive algorithm for computing products in $ \mc{P}_G. $

\subsection{$ a\neq 1 $ case} Suppose $ a\neq 1 $ now. By Lemma \ref{critical}, all candidates of block generators for coinvariants of $ G(m,a,n) $ are $ \alpha_{\sigma}= x_{\sigma(1)}^{m-1}x_{\sigma(2)}^{2m-1}\cdots x_{\sigma(n)}^{mn/a-1} $ for $ \sigma\in S_n. $

\begin{lemma}\label{2}
	Keep notations as above. $ \alpha_{\sigma}\in \mc{I}. $ In particular, there is no block basis for the coinvariants of $ G(m,a,n). $
\end{lemma}
\begin{proof}
	Note that $ mn/a-1\geq m/a $ unless $ n=1 $ which we omit. Since $ a\neq 1 $ by hypothesis, hence $ (m-1)(a-1)\geq 1 $ and $ m-1\geq m/a. $ Therefore, there exist $ f_n\in \mc{P}$ such that $$ \alpha_{\sigma}=(x_1\cdots x_n)^{m/a}f_n=e_nf_n\in \mc{I}. $$
\end{proof}

\subsection{general case}
Combining Corollary \ref{1}, Proposition \ref{5.4} and Lemma \ref{2}, we have one of our main results.
\begin{theorem}\label{5.6}
	The coinvariants of $ G(m,a,n) $ has block basis if and only if $ a=1. $ 
	Moreover, if $ a=1, $ the monomial $ x_1^{mn-1}x_2^{m(n-1)-1}\cdots x_n^{m-1} $ generates a block basis with products $$ x_{n+1-r}^{mr}=-\sum_{\stackrel{j_1+\cdots+j_n= r}{j_n\neq r}} x_{j_1}^{mj_1}\cdots x_{j_{n+1-r}}^{mj_{n+1-r}}. $$
\end{theorem}

\section{block basis for $ \mc{P}_{U_I} $}
Denote $ H_n=\ff_q[x_1,\cdots,x_n]^{U_I}. $ In this section, we will prove that there is a block basis for $ \mc{P}_{U_I}:=\mc{P}/(H_{n,+}). $

\subsection{} 
It will be useful introduce an algebra homomorphism $ \theta: \ff_q[x_1,\cdots,x_n] \ra \ff_q[x_1,\cdots,x_{m_{l-1}}] $ by $ \theta(x_i)=x_i $ for $ 1\leq i\leq m_{l-1} $ and $ \theta(x_i)=0 $ otherwise. Then $ \ker(\theta) $ is the ideal generated by $ x_{m_{l-1}+1},\cdots, x_n. $

\subsection{}Let $ \Psi_n=\{ x_1^{i_1}\cdots,x_n^{i_n}\mid \forall k,\ 0\leq i_k <\tau(k) \}. $ We will prove that $ \Psi_n $ is a basis for $ \mc{P}$ over $ \mc{P}^{U_I}. $ 
Namely, 
\begin{equation}
	f_I:=\prod_{i=1}^{n}x_i^{q^{\tau(i)}-1}= x_{m_1+1}^{q^{m_1}-1}\cdots x_{m_2}^{q^{m_1}-1}x_{m_2+1}^{q^{m_2}-1} \cdots x_{m_l}^{q^{m_{l-1}}-1}
\end{equation}
generates a block basis for $\mc{P}_{U_I}. $

\begin{lemma}\label{4.1}
	If $ \alpha\in \Psi_n, $ then $ x_{m_{l-1}+j}\alpha $ is in the $ \mc{P}^{U_I} $-module spanned by $ \Psi_n $ for all $ j=1,\cdots,n_l. $
\end{lemma}
\begin{proof}
	Suppose $ \alpha=x_1^{i_1}\cdots,x_n^{i_n}\in \Psi_n. $ Without loss of generality, one can assume $ j=n_l, $ and hence $ m_{l-1}+j=n. $ 
	
	If $ i_n< q^{m_{l-1}}-1, $ then $ x_n\alpha\in \Psi_n $ and Lemma holds. Now, if $ i_n= q^{m_{l-1}}-1, $ then $ x_n\alpha=x_1^{i_1}\cdots x_{n-1}^{i_{n-1}}x_n^{q^{m_{l-1}}}. $  By \ref{vij}, \[ v_{l,n_l} = F_{m_{l-1}}(x_n)=x_n^{q^{m_{l-1}}} +\sum_{i=1}^{m_{l-1}}Q_{m_{l-1},i} x_n^{q^{m_{l-1}-i}}. \]
	Therefore, \[  x_n^{q^{m_{l-1}}}= v_{l,n_l} -\sum_{i=1}^{m_{l-1}}Q_{m_{l-1},i} x_n^{q^{m_{l-1}-i}}, \]
	and $$ x_n\alpha= v_{l,n_l}x_1^{i_1}\cdots x_{n-1}^{i_{n-1}} -\sum_{i=1}^{m_{l-1}}Q_{m_{l-1},i} x_1^{i_1}\cdots x_{n-1}^{i_{n-1}}x_n^{q^{m_{l-1}-i}}. $$
	
	Since $ Q_{m_{l-1},i}\in  \mc{P}^{G_I} \subseteq \mc{P}^{U_I} $ and $ v_{l,n_l}\in \mc{P}^{U_I}. $ Lemma holds.
\end{proof}

\begin{theorem}\label{U_I}
	$ \Psi_n $ is a basis for $ \mc{P}$ over $ \mc{P}^{U_I}. $ Namely, $ f_I $ generates a block basis  for $\mc{P}_{U_I}. $
\end{theorem}
\begin{proof}
	Since the cardinality of $ \Psi_n $ equals the dimension of $ \mc{P}_{U_I}, $ it is
	sufficient to prove that $ \Psi_n $ is a spanning set. We use induction on $ n. $
	
	For $ n=1, $ the result is trivial.
	
	For $ n>1, $ we use induction on the degree. The result is clear for polynomials of degree zero. Now, suppose $ f\in \mc{P}$ with $ \deg(f)>0. $ By the primary induction hypothesis $$ \theta(f)=\sum_{\beta\in \Psi_{m_{l-1}}} c_{\beta}\beta $$ for some  $ c_{\beta}\in H_{m_{l-1}}. $ The kernel of $ \theta $ is the ideal generated by $ x_{m_{l-1}+1},\cdots, x_n $ and therefore $ f=\theta(f) +x_{m_{l-1}+j}\pr{f} $ for some $ m_{l-1}<j\leq n $ and $ \pr{f}\in \mc{P}. $ Since $ \deg(\pr{f}) <\deg(f), $ by the secondary induction hypothesis, 
	\[ f=\theta(f) +x_{m_{l-1}+j}\pr{f}=\sum_{\beta\in \Psi_{m_{l-1}}} c_{\beta}\beta+\sum_{\gamma\in \Psi_{n}} d_{\gamma}x_{m_{l-1}+j}\gamma. \]
	
	By Lemma \ref{4.1}, $ x_{m_{l-1}+j}\gamma $ is in the span of $ \Psi_n. $ Note that $ \Psi_{m_{l-1}}\subseteq\Psi_n $ and $ H_{m_{l-1}}\subseteq H_n.$ Therefore, $ f$ is in the $ H_n $ span of $ \Psi_n. $ 
\end{proof}

\subsection{products of block basis}
\begin{lemma}\label{5.3}
	For each $ 0\leq i\leq l-1,\ 1\leq j\leq n_i, $
	$ x_{m_{i}+j} ^{q^{m_{i}}}\equiv 0 $ in $ \mc{P}_{U_I}. $
\end{lemma}
\begin{proof}
	By \ref{vij}, $ v_{i+1,j} = F_{m_{i}}(x_{m_{i}+j})=x_{m_{i}+j} ^{q^{m_{i}}} +\sum_{k=1}^{m_{i}}Q_{m_{i},k} x_{m_{i}+j}^{q^{m_{i}-k}}. $
	All $ v_{i+1,j} $ and $ Q_{m_{i},k} $ lie in $ \mc{I}. $ Hence lemma holds.
\end{proof}
\begin{corollary}\label{5.41}
	$ \mc{P}_{U_I}\iso \bten_{i=1}^n \fk[x_i]/(x_i^{q^{\tau(i)}}). $
\end{corollary}

\begin{remark}
Note that the hypothesis of Lemma \ref{critical2} are satisfied by monomial $ f_I. $ This gives another proof of the existence of a block basis for $ \mc{P}_{U_I}. $ In fact,  $ \Sigma(H_n)=\{1\}. $ And the critical monomials associated to $ f_I $ are $ x_j^{q^{\tau(j)}}. $ By \ref{5.3}, $ x_j^{q^{\tau(j)}}\equiv 0. $
Moreover, by \cite[Theorem 4.7]{CHSW}, one get Corollary \ref{5.41}.
\end{remark}

\section{Block basis for $ \mc{P}_{G_I} $}
\subsection{}
For $ 1\leq k\leq l, $ denote $ \tilde{G}_k=\left( 
\begin{matrix}G_1 & * & \cdots & *\\ 
0 & G_2 &\cdots & *\\
\vdots & \vdots & \vdots & \vdots\\
0 & 0 & \cdots & G_k
\end{matrix}
\right), \ \mc{Q}_k=\ff_q[x_1,\cdots,x_{m_k}], $ $ \mc{I}_k= \left(\mc{Q}_{k,+}^{\tilde{G}_k}\right) $ and $ \mc{P}_k=\ff_q[x_{m_{k-1}+1}, \cdots,x_{m_{k-1}+n_k}]. $

For each $ 1\leq i\leq l, $ suppose $ x_{m_{i-1}+1}^{\alpha_{i1}}\cdots x_{m_i}^{\alpha_{i,n_i}} $ generates a block basis for $ \mc{P}_i $ over $ \mc{P}_i^{G_i}. $

\begin{lemma}
	Keep notations as above. $ \mc{P} $ is a free $ \mc{P}^{G_I} $ module with a basis consisting of 
	$$ \left\{\prod_{k=1}^{n}x_k^{a_k}\prod_{i=1}^{l}\prod_{j=1}^{n_i} v_{ij} ^{b_{ij}}\ \vline\ 0\leq a_k <\tau(k),\ 0\leq b_{ij}<\alpha_{ij} \right\}.$$
\end{lemma}
\begin{proof}
 Let $ \mc{R}_i=\ff_q[v_{i,1},\cdots,v_{i,n_i}] $ be a polynomial ring with generators $ v_{ij},\ 1\leq j\leq n_i. $ Then
\[  \mc{R}_i =\bds_{j=1}^{n_i}\bds_{b_{ij}=1}^{\alpha_{ij}} v_{i,1}^{b_{i1}}\cdots v_{i,n_i}^{b_{i,n_i}}(\mc{R}_i) ^{G_i}.  \]

By Theorem \ref{U_I}, $ f_I=\prod_{i=1}^{n}x_i^{q^{\tau(i)}-1} $ generates a block basis for $ \mc{P}_{U_I}. $ Equivalently,
\[ \mc{P}=\bds_{i=1}^n\bds_{a_i=0}^{\tau(i)-1} x_1^{a_1} \cdots x_n^{a_n}\mc{P}^{U_I}. \]

Since $ \mc{P}^{U_I}=\mc{R}_1\ten \cdots \ten \mc{R}_l, $ lemma holds.
\end{proof}

\begin{lemma}
For all $ 1\leq i\leq l, 1\leq j\leq n_i,\ v_{ij}\equiv x_{m_{i-1}+j}^{q^{m_{i-1}}} .$
\end{lemma}
\begin{proof}
		By \ref{vij}, $ v_{i,j}=x_{m_{i-1}+j} ^{q^{m_{i-1}}} +\sum_{k=1}^{m_{i-1}}Q_{m_{i-1},k} x_{m_{i-1}+j}^{q^{m_{i-1}-k}}. $ Since $ Q_{m_{i-1},k}\in \mc{I}_{i-1}\subseteq \mc{I}, $ lemma holds.
\end{proof}

\begin{theorem}\label{6.4}
Suppose $ x_{m_{i-1}+1}^{\alpha_{i1}}\cdots x_{m_i}^{\alpha_{i,n_i}} $ generates a block basis for $ \mc{P}_i $ over $ \mc{P}_i^{G_i}. $ Then
$$ \prod_{i=1}^{l}\prod_{j=1}^{n_i} x_{m_{i-1}+j}^{(\alpha_{ij}+1)q^{m_{i-1}}-1} $$ generates a block basis of $ \mc{P} $ over $ \mc{P}^{G_I}. $
\end{theorem}

\subsection{Block basis for $ \mc{P}_{\GL_I} $}
If $ G_I=\GL_I, $ then $ G_i=\GL_{n_i}(q). $ By \cite[Theorem B]{St} (a generalization refers to \cite[Theorem 3.2]{CHSW}), the monomial $ \prod_{j=1}^{n_i}x_{m_{i-1}+j}^{q^{n_i}- q^{n_i-j}-1} $ generates a block basis for $ \mc{P}_i $ over $ \mc{P}_i^{\GL_{n_i}(p)}. $ The following proposition is a corollary of Theorem \ref{6.4}.
\begin{proposition}\label{6.5}
	The monomial $ \prod_{i=1}^{l}\prod_{j=1}^{n_i} x_{m_{i-1}+j}^{q^{m_i}- q^{m_i-j}-1} $ generates a block basis for $ \mc{P}_{\GL_I}. $
\end{proposition}
\begin{remark}
(1) Take $ I=(n), $ then $ l=1, n_1=m_1=n. $ Proposition provides a block basis for $ \mc{P}_{\GL_n(q)} $ which coincide with \cite[Theorem B]{St} and \cite[Theorem 3.2]{CHSW}.

(2) Take $ I=(1,\cdots,1), $ then $ l=n, m_i=i $ and the monomial $ \prod_{i=1}^{l} x_{i}^{q^{i}- q^{i-1}-1} $ generates a block basis for $ \mc{P}_{B} $ where $ B $ is a Borel subgroup of $ \GL_n(q) $ consisting of all upper triangle matrices.
\end{remark}

\subsection{Weyl groups of Cartan type Lie algebras}
By \cite{Je}, $ G_I $ is a Weyl group of restricted Cartan type Lie algebra $ \mf{g} $ of type $ W, S $ or $ H $ where  $ I=(n_1,n_2), $ 
$ G_I=\left\{
\left(\begin{matrix}
A & B\\ 0 &C
\end{matrix}\right)\ 
\vline\ A\in \GL_{n_1}(p),\ C\in G_2 \right\}, $
$ G_2=G(m,1,n_2) \text{ and } m=
\left\{
\begin{array}{ll}
1 &\text{ if }\lieg\text{ is of type } W \text{ or } S, \\
2 &\text{ if } \mf{g} \text{ is of type }  H.  
\end{array}
\right.
$
\iffalse
Recall that $ P^{U_I}=\ff_q[x_1,\cdots,x_{n_1},v_{1,1},\cdots v_{1,n_2}] $ and
$ \ff_q[x_{n_1+1},\cdots,x_{n}]^{S_{n_2}}=\ff_q[e_1,\cdots,e_{n_2}] $ where $ e_j=\sum_{n_1+1\leq i_1<\cdots<i_j\leq n} x_{i_1}\cdots x_{i_j}. $ Define $ u_i=e_i(v_{1,1},\cdots v_{1,n_2}). $ 
\fi

 By Theorem \ref{5.6}, the monomial $ \prod_{i=n_1+1}^{n} x_i^{m(n-i+1)-1} $ generates a block basis for $ \mc{P}_2 $ over $ \mc{P}_2^{G_2}. $ 
 %where $ m=1 $ if $ G_2=S_{n_2} $ and $ m=2 $ if $ G_2=G(2,1,n_2). $

The following is a corollary of Theorem \ref{6.4}.
\begin{proposition}\label{6.7}
	The monomial $ x_1^{p^{n_1}-p^{n_1-1}-1}\cdots x_{n_1}^{p^{n_1}-2} x_{n_1+1}^{m(n_2+1)p^{n_1}-1}x_{n_1+2}^{mn_2p^{n_1}-1}\cdots x_n^{mp^{n_1}-1} $ generates a block basis for $ \mc{P} $ over $ \mc{P}^{G_I}. $ 
	%where $ m=1 $ if $ G_2=S_{n_2} $ and $ m=2 $ if $ G_2=G(2,1,n_2). $
\end{proposition}
%\bigskip
\noindent
\textbf{Acknowledgments.}
This work is supported by Fundamental Research Funds of Yunnan Province (No. 2020J0375), the Fundamental Research Funds of YNUFE (No. 80059900196).


\begin{thebibliography}{}
\bibitem{CHSW} H. E. A. Campbell, I. P. Hughes, R. J. Shank and D. L. Wehlau, {\em Bases for rings of coinvariants,}  Transformation Groups, Vol.1, No.4(1996), 307-336. 

\bibitem{Di} L. Dickson, {\em A fundamental system of invariants of the general modular linear group with a solution of the form problem,} Trans. Amer. Math. Soc. 12 (1911), 75-98.

\bibitem{He} T. Hewett, {\em Modular invariant theory of parabolic subgroups of $ \on{GL}_n(\ff_q) $ and the associated steenrod modules,} Duke Math. J. 82 (1996), 91-102; Erratum, Duke Math. J. 97 (1999), 217.

\bibitem{Jan1} J.Jantzen, {\em Representations of Lie algebras in prime characteristic}, in {\em Representation Theories and Algebraic Geometry}, Proceedings, Montreal, 1997, NATO ASI Series, Vol. C514, Kluwer, Dordrecht, (1998), 185-235.

\bibitem{Je} M. Jensen, {\em Invariant Theory of Restricted Cartan Type Lie Algebras}, PhD-Thesis, Aarhus University, 2015.

\bibitem{Ka} R. Kane, {\em Reflection Group and Invariant Theory,} CMS Books in Mathematics 5, Springer-Verlag, New York, 1997.

\bibitem{KM} N. Kuhn, S. Mitchell, {\em The multiplicity of the Steinberg representation of $ \on{GL}_n(\ff_q) $ in the symmetric algebra,} Proc. Amer. Math. Soc. 96 (1986), 1-6.

\bibitem{MT} P. Minh, V. T\`{u}ng, {\em Modular invariants of parabolic subgroups of general linear groups,} J. Algebra 232 (2000), 197-208.

\bibitem{Na} H.Nakajima, {\em Invariants of finite groups generated by  pseudo-reflections in positive characteristic,} Tsukuba J. Math, Vol. 3, No. 1 (1979), 109-122.

\bibitem{Ou} K. Ou, {\em Modular Invariants for some Finite Pseudo-Reflection Groups}, arXiv:2003.03557.

\bibitem{St} R. Steinberg, {\em On Dickson's theorem on invariants, } J. Fac. Sci. Univ. Tokyo. 34 (1987), 699-707.
\end{thebibliography}
\end{document}